\documentclass[12pt]{article}
\usepackage{amssymb}
\usepackage{amsmath}
\usepackage{amsthm}

\begin{document}
\newtheorem{Thm}{Theorem}[section]
\newtheorem{Def}{Definition}[section]
\newtheorem{Lem}{Lemma}[section]
\newtheorem{Prop}{Proposition}[section]
\newtheorem{Cor}{Corollary}[section]

\theoremstyle{definition}
\newtheorem*{remark}{Remark}

\newcommand{\Z}{{\mathbf Z}}
\newcommand{\R}{{\mathbf R}}
\newcommand{\Q}{{\mathbf Q}}
\newcommand{\C}{{\mathbf C}}
\newcommand{\N}{{\mathbf N}}
\newcommand{\eop}{\hfill$\square$}

\title{Depth reduction of a class of Witten zeta functions}
\author{Xia Zhou\thanks{The first and third authors are supported
by the National Natural Science Foundation of China, Project 10871169.}\\
\small Department of mathematics\\[-0.8ex]
\small Zhejiang University\\[-0.8ex]
\small Hangzhou, 310027\\[-0.8ex]
\small P.R.China\\
\small \texttt{xiazhou0821@hotmail.com}\\
\and
David~M. Bradley\\
\small Department of Mathematics \& Statistics\\[-0.8ex]
\small University of Maine\\[-0.8ex]
\small 5752 Neville Hall
         Orono, Maine 04469-5752\\[-0.8ex]
\small U.S.A\\
\small \texttt{bradley@math.umaine.edu, dbradley@member.ams.org}
\and
Tianxin Cai\\
\small Department of mathematics\\[-0.8ex]
\small Zhejiang University\\[-0.8ex]
\small Hangzhou, 310027\\[-0.8ex]
\small P.R.China\\
\small \texttt{caitianxin@hotmail.com}
}

\date{Submitted April 6, 2008; Accepted July 21, 2009; Published July 31, 2009}

\maketitle
\begin{abstract}
We show that if $a,b,c,d,f$ are positive integers such that
$a+b+c+d+f$ is even, then the Witten zeta value
$\zeta_{\mathfrak{sl(4)}}(a,b,c,d,0,f)$ is expressible in terms of
Witten zeta functions with fewer arguments.
\end{abstract}

\section{Introduction}
Let $\N$ be the set of positive integers, $\Q$ the field of rational
numbers, $\C$ the field of complex numbers.

 For any semisimple Lie algebra $\mathfrak{g}$, the Witten zeta
function(cf.~\cite{MT}) is defined by
\begin{equation*}
\zeta_{\mathfrak{g}}(s)=\sum_{\rho}(\text{dim}\ \rho)^{-s},
\end{equation*}
where $s\in \C$ and $\rho$ runs over all finite dimensional
irreducible representations of $\mathfrak{g}$. In order the
calculate the volumes of certain moduli space, Witten~\cite{Witten}  introduced
the values $\zeta_{\mathfrak{g}}(2k)$ for $k\in \N$ and showed that
$\pi^{-2kl}\zeta_{\mathfrak{g}}(2k)\in \Q$, where $l$ is the number
of  positive roots of $\mathfrak{g}$.

For positive integer $r$, Matsumoto and Tsumura~\cite{MT} defined a
multi-variate extension, called the Witten multiple zeta-function
associated with $\mathfrak{sl}(r+1)$, by
\begin{equation}\label{MWit}
\zeta_{\mathfrak{sl}(r+1)}(\textbf{s})=\sum_{m_1, \dots,
m_r=1}^{\infty}\prod_{j=1}^r\prod_{k=1}^{r-j+1}\bigg(\sum_{v=k}^{j+k-1}m_v\bigg)^{-s_{j,k}}
\end{equation}
where \begin{equation*} \textbf{s}=(s_{j,k})_{1\leq j\leq r;\ 1\leq
k\leq r-j+1}\in \C^{r(r+1)/2}, \quad \Re (s_{j,k})>1.\end{equation*}
In particular (\cite{MT}, section 2, Proposition 2.1), if $m\in \N$
we denote
\[\zeta_{\mathfrak{sl}(r+1)}(2m):=\prod_{1\leq j<k\leq
r+1}(k-j)\zeta_{\mathfrak{sl}(r+1)}(\,\underbrace{2m,\dots,
2m}_{r(r+1)/2}\,).\] As in~\cite{BBBP}, given the Witten multiple
zeta-function~\eqref{MWit}, we define the \textit{depth} to be $r$. Further,
if the zeta functions $y_1,\dots,y_k$ have depth $r_1,\dots,r_k$
respectively, then for $a_1,\dots,a_k\in\C$, we define the \textit{depth} of
$a_1y_1+\cdots+\ a_ky_k$ to be $\max\{r_i : 1\leq i\leq k\}$. We
would like to know which sums can be expressed in terms of lower
depth sums. When a sum can be so expressed, we say it is \textit{reducible}.

An explicit evaluation for $\zeta_{\mathfrak{sl}(3)}(2m)$
$(m\in\N)$  was independently discovered by D.~Zagier,
S.~Garoufalidis, and L.~Weinstein (see \cite[page 506]{Zag}). In
\cite{Gun}, Gunnells and Sczech provided a generalization of the
continued-fraction algorithm to compute high-dimensional Dedekind
sums. As examples, they gave explicit evaluations of
$\zeta_{\mathfrak{sl}(3)}(2m)$ and $\zeta_{\mathfrak{sl}(4)}(2m)$.
Matsumoto and Tsumura~\cite{MT} considered functional relations for
Witten multiple zeta-functions, and found that
\begin{multline}\label{reducible}
 (-1)^a\zeta_{\mathfrak{sl(4)}}(s_1,s_2,a,s_3,0,b)+(-1)^b\zeta_{\mathfrak{sl(4)}}(s_1,s_2,b,s_3,0,a)\\
+\zeta_{\mathfrak{sl(4)}}(a,0,s_2,s_1,b,s_3)+\zeta_{\mathfrak{sl(4)}}(b,0,s_1,s_2,a,s_3)
\end{multline}
is reducible for any $a,b\in \N$ and $s_1, s_2, s_3\in \C$.

In this paper, we provide a combinatorial method which gives a
simpler formula for the quantity~\eqref{reducible}. Furthermore, we
show that if $a,b,c,d,f$ are positive integers such that $a+b+c+d+f$
is even, then $\zeta_{\mathfrak{sl(4)}}(a,b,c,d,0,f)$ is reducible.

\section{Functional relation}
\begin{Lem}\label{lem:2.1}
If the function $F:\Z_{\ge 0}\times \Z_{\ge 0}\times \C\to\C$
has the property that there exist $p,\ q\in \C$  such that for every  $a, b\in \N$ and every $s\in \C$ the relation
\[F(a,b,s)=pF(a-1,
b,s+1)+qF(a,b-1,s+1)\] holds, then for every  $a, b\in \N$ and every $s\in \C$,
\begin{align}\label{Decomp.}
F(a,b,s)=&\sum_{j=1}^{b}p^aq^{b-j}\binom{a+b-j-1}{a-1}F(0,j,a+b+s-j)\nonumber\\
+&\sum_{j=1}^{a}p^{a-j}q^{b}\binom{a+b-j-1}{b-1}F(j,0,a+b+s-j).
\end{align}
\end{Lem}

\begin{proof}
It's easy to prove Lemma~\ref{lem:2.1} by induction.
\end{proof}

The Euler sum of depth $r$ and weight $w$ is a multiple series of
the form
\begin{equation}\label{MZVdef}
   \zeta(s_1,\dots,s_r) := \sum_{n_1>\cdots>n_r>0}\; \prod_{j=1}^r
   n_j^{-s_j},
\end{equation}
with \textit{weight} $w:=s_1+\cdots+s_r$. Now let's recall the following result
concerning the reduction on the triple Euler sums.

\begin{Lem}[Borwein-Girgensohn~\cite{BG}]\label{lem:2.2}
Let $a,b,c$ be positive integers. If $a+b+c$ is even or less than or
equal to 10, then $\zeta(a,b,c)$ can be expressed as a rational
linear combination of products of single and double Euler sums of
weight $a+b+c$.
\end{Lem}

\begin{Lem}[Huard-Williams-Zhang~\cite{HWZ}]\label{lem:2.3}
If $a,b,c$ be positive integers, then
\begin{equation}\label{MT2}
\zeta_{\mathfrak{sl(3)}}(a,b,c)=\Bigg\{\sum_{j=1}^{a}\binom{a+b-j-1}{b-1}
   +\sum_{j=1}^{b}\binom{a+b-j-1}{a-1}\Bigg\}\zeta(a+b+c-j,j).
   \end{equation}
Moreover, $\zeta_{\mathfrak{sl(3)}}(a,b,c)$ can be explicitly
evaluated in terms of the values of Riemann zeta functions when
$a+b+c$ is odd.
\end{Lem}

\begin{Thm}\label{thm:2.1} If $a,b\in \N$, then
\begin{align}
&(-1)^a\zeta_{\mathfrak{sl(4)}}(s_1,s_2,a,s_3,0,b)+(-1)^b\zeta_{\mathfrak{sl(4)}}(s_1,s_2,b,s_3,0,a)\notag\\
&\hspace{3cm}+\zeta_{\mathfrak{sl(4)}}(a,0,s_2,s_1,b,s_3)+\zeta_{\mathfrak{sl(4)}}(b,0,s_1,s_2,a,s_3)\notag\\
&=\sum_{i=1}^{\max(a,b)}\Bigg\{\binom{a+b-i-1}{a-1}+\binom{a+b-i-1}{b-1}\Bigg\}(-1)^{i}\zeta(i)\notag\\
&\hspace{3cm}\times\zeta_{\mathfrak{sl(3)}}(s_1,s_2,s_3+a+b-i)\notag\\
&\hspace{1cm}+\sum_{i=1}^{a}\binom{a+b-i-1}{b-1}\Bigg\{\zeta(i)\zeta_{\mathfrak{sl(3)}}(s_1,s_2,s_3+a+b-i)\notag\\
&\hspace{2.5cm}-\zeta_{\mathfrak{sl(3)}}(s_1+i,s_2,s_3+a+b-i)-\zeta_{\mathfrak{sl(3)}}(s_1,s_2,s_3+a+b)\Bigg\}\notag\\
&\hspace{1cm}+\sum_{i=1}^{b}\binom{a+b-i-1}{a-1}\Bigg\{\zeta(i)\zeta_{\mathfrak{sl(3)}}(s_1,s_2,s_3+a+b-i)\notag\\
&\hspace{2.5cm}-\zeta_{\mathfrak{sl(3)}}(s_2+i,s_1,s_3+a+b-i)-\zeta_{\mathfrak{sl(3)}}(s_1,s_2,s_3+a+b)\Bigg\}.
\end{align}
\end{Thm}

\begin{proof} From the definition~\eqref{MWit} of the Witten multiple
zeta-function, we have
\begin{equation}
\zeta_{\mathfrak{sl(4)}}(s_1,s_2,s_3,s_4,s_5,s_6)=\zeta_{\mathfrak{sl(4)}}(s_3,s_2,s_1,s_5,s_4,s_6).
\end{equation}
Next, for any $a,b\in \N$ and $s_1,s_2,s_3\in \C$, since
\begin{multline*}
   \zeta_{\mathfrak{sl(4)}}(s_1,s_2,a,s_3,0,b)=\zeta_{\mathfrak{sl(4)}}(s_1,s_2,a,s_3+1,0,b-1)\\
   -\zeta_{\mathfrak{sl(4)}}(s_1,s_2,a-1,s_3+1,0,b),
\end{multline*}
by Lemma~\ref{lem:2.1}, we have
\begin{align}
\zeta_{\mathfrak{sl(4)}}(s_1,s_2,a,s_3,0,b)&=\sum_{i=1}^{a}\binom{a+b-i-1}{b-1}(-1)^{a+i}\zeta_{\mathfrak{sl(4)}}(s_1,s_2,i,s_3+a+b-i,0,0)\notag\\
&+\sum_{i=1}^{b}\binom{a+b-i-1}{a-1}(-1)^{a}\zeta_{\mathfrak{sl(4)}}(s_1,s_2,0,s_3+a+b-i,0,i).
\end{align}
Similarly, we have
\begin{align}
\zeta_{\mathfrak{sl(4)}}(s_1,s_2,b,s_3,0,a)&=\sum_{i=1}^{b}\binom{a+b-i-1}{a-1}(-1)^{b+i}\zeta_{\mathfrak{sl(4)}}(s_1,s_2,i,s_3+a+b-i,0,0)\notag\\
&+\sum_{i=1}^{a}\binom{a+b-i-1}{b-1}(-1)^{b}\zeta_{\mathfrak{sl(4)}}(s_1,s_2,0,s_3+a+b-i,0,i),
\end{align}
\begin{align}
\zeta_{\mathfrak{sl(4)}}(a,0,s_2,s_1,b,s_3)&=\sum_{i=1}^{a}\binom{a+b-i-1}{b-1}\zeta_{\mathfrak{sl(4)}}(i,0,s_2,s_1,0,s_3+a+b-i)\notag\\
&+\sum_{i=1}^{b}\binom{a+b-i-1}{a-1}\zeta_{\mathfrak{sl(4)}}(0,0,s_2,s_1,i,s_3+a+b-i),
\end{align}
and
\begin{align}
\zeta_{\mathfrak{sl(4)}}(b,0,s_1,s_2,a,s_3)&=\sum_{i=1}^{b}\binom{a+b-i-1}{a-1}\zeta_{\mathfrak{sl(4)}}(i,0,s_1,s_2,0,s_3+a+b-i)\notag\\
&+\sum_{i=1}^{a}\binom{a+b-i-1}{b-1}\zeta_{\mathfrak{sl(4)}}(0,0,s_1,s_2,i,s_3+a+b-i).
\end{align}
Since
\begin{align}
\zeta_{\mathfrak{sl(4)}}(a,b,c,d,0,0)&=\zeta(c)\zeta_{\mathfrak{sl(3)}}(a,b,d),\\
\zeta_{\mathfrak{sl(4)}}(a,b,0,c,0,d)&=\sum_{\substack{n_1,n_2=1\\v>n_1+n_2}}\frac{1}{v^dn_1^{a}n_2^{b}(n_1+n_2)^{c}}\notag\\
&=\sum_{\substack{n_1,n_2=1\\v>n_1+n_2}}\frac{1}{v^dn_1^{b}n_2^{a}(n_1+n_2)^{c}},\\
\zeta_{\mathfrak{sl(4)}}(a,0,b,c,0,d)&=\sum_{\substack{n_1,n_2=1\\v<n_1}}\frac{1}{v^an_1^cn_2^b(n_1+n_2)^d},\\
\zeta_{\mathfrak{sl(4)}}(0,0,a,b,c,d)&=\sum_{\substack{n_1,n_2=1\\n_1+n_2>v>n_1}}\frac{1}{v^cn_1^an_2^{b}(n_1+n_2)^d},
\end{align}
we find that
\begin{align}
&\zeta_{\mathfrak{sl(4)}}(s_1,s_2,0,s_3+a+b-i,0,i)+\zeta_{\mathfrak{sl(4)}}(i,0,s_2,s_1,0,s_3+a+b-i)\nonumber\\
&\hspace{4cm}+\zeta_{\mathfrak{sl(4)}}(0,0,s_1,s_2,i,s_3+a+b-i)\nonumber\\
&=\zeta(i)\zeta_{\mathfrak{sl(3)}}(s_1,s_2,s_3+a+b-i)-\zeta_{\mathfrak{sl(3)}}(s_1+i,s_2,s_3+a+b-i)\nonumber\\
&\hspace{4cm}-\zeta_{\mathfrak{sl(3)}}(s_1,s_2,s_3+a+b)
\end{align}
and
\begin{align}
&\zeta_{\mathfrak{sl(4)}}(s_1,s_2,0,s_3+a+b-i,0,i)+\zeta_{\mathfrak{sl(4)}}(i,0,s_1,s_2,0,s_3+a+b-i)\notag\\
&\hspace{4cm}+\zeta_{\mathfrak{sl(4)}}(0,0,s_2,s_1,i,s_3+a+b-i)\notag\\
&=\zeta(i)\zeta_{\mathfrak{sl(3)}}(s_1,s_2,s_3+a+b-i)-\zeta_{\mathfrak{sl(3)}}(s_2+i,s_1,s_3+a+b-i)\notag\\
&\hspace{4cm}-\zeta_{\mathfrak{sl(3)}}(s_1,s_2,s_3+a+b)
\end{align}
Now combining equations (8-17), we complete
the proof.
\end{proof}

\begin{Lem}\label{lem:2.4} Every Witten multiple zeta value of the form $\zeta_{\mathfrak{sl(4)}}(a,b,1,d,0,1)$ with
$a,b,d\in \N$ can be expressed as a rational linear combination of
products of single and double Euler sums when $a+b+d$ is even or $a+b+d\le 8$.
\end{Lem}
\begin{proof}
\begin{align}
\zeta_{\mathfrak{sl(4)}}(a,b,1,d,0,1)&=\sum_{i=1}^{a}\binom{a+b-i-1}{b-1}\zeta_{\mathfrak{sl(4)}}(i,0,1,a+b+d-i,0,1)\notag\\
&+\sum_{i=1}^{b}\binom{a+b-i-1}{a-1}\zeta_{\mathfrak{sl(4)}}(0,i,1,a+b+d-i,0,1).
\end{align}
However, for any $a,d\in \N$,
\begin{align}
\zeta_{\mathfrak{sl(4)}}(a,0,1,d,0,1)&=\zeta_{\mathfrak{sl(4)}}(0,a,1,d,0,1)\notag\\
&\hspace{-1cm}=\zeta_{\mathfrak{sl(4)}}(a,0,1,0,0,d+1)+\sum_{i=1}^{d}\zeta(d+2-i,i,a),
\end{align}
and
\begin{align}\label{19}
\zeta_{\mathfrak{sl(4)}}(a,0,1,0,0,d+1)&=\zeta(d+1,a,1)+\sum_{i=1}^{a}\zeta(d+1,a+1-i,i).
\end{align}
We complete the proof by combining this with Lemma 2.2.
\end{proof}
\begin{Thm}\label{thm:2.2} Every Witten multiple zeta value of the form $\zeta_{\mathfrak{sl(4)}}(a,b,c,d,0,f)$   with $a,b,c,d,f,\in \N$
 can be expressed as a rational linear combination of products of
single and double Euler sums when  $a+b+c+d+f$ is even or
$a+b+c+d+f\leq 10$.
\end{Thm}
\begin{proof}
From Lemma~\ref{lem:2.1}, we see that
\begin{align}\label{20}
&\frac{1}{n_1^an_2^bn_3^c(n_1+n_2)^d(n_1+n_2+n_3)^f}\notag\\
&\hspace{2cm}=\sum_{i=1}^{c}\binom{c+f-i-1}{f-1}(-1)^{c+i}\frac{1}{n_1^an_2^bn_3^i(n_1+n_2)^{c+d+f-i}}\notag\\
&\hspace{2cm}+\sum_{i=1}^{f}\binom{c+f-i-1}{c-1}(-1)^{c}\frac{1}{n_1^an_2^b(n_1+n_2)^{c+d+f-i}(n_1+n_2+n_3)^i}.
\end{align}
Also
\begin{align}&\frac{1}{n_1^an_2^b(n_1+n_2)^{c+d+f-i}(n_1+n_2+n_3)^i}\notag\\
&\hspace{2cm}=\sum_{j=1}^{a}\binom{a+b-j-1}{b-1}\frac{1}{n_1^j(n_1+n_2)^{a+b+c+d+f-i-j}(n_1+n_2+n_3)^i}\notag\\
&\hspace{2cm}+\sum_{j=1}^{b}\binom{a+b-j-1}{a-1}\frac{1}{n_2^j(n_1+n_2)^{a+b+c+d+f-i-j}(n_1+n_2+n_3)^i}.
\end{align}
Now combine~\eqref{19},~\eqref{20} and Lemma~\ref{lem:2.4} and sum
over all ordered triples of positive integers $(n_1, n_2, n_3)$ to
obtain
\begin{align}
&\zeta_{\mathfrak{sl(4)}}(a,b,c,d,0,f)=\sum_{i=2}^{c}\binom{c+f-i-1}{f-1}(-1)^{c+i}\zeta(i)\zeta_{\mathfrak{sl}(3)}(a,b,c+d+f-i)\notag\\
&\hspace{1.5cm}+\sum_{i=2}^{f}\binom{c+f-i-1}{c-1}(-1)^{c}\Bigg\{\sum_{j=1}^{a}\binom{a+b+j-1}{b-1}\notag\\
&\hspace{4cm}\times\zeta(i,c+d+f+a+b-i-j,j)\notag\\
&\hspace{3cm}+\sum_{j=1}^{b}\binom{a+b+j-1}{a-1}\zeta(i,c+d+f+a+b-i-j,j)\Bigg\}\notag\\
&\hspace{3cm}-(-1)^{c}\binom{c+f-2}{c-1}\zeta_{\mathfrak{sl(4)}}(a,b,1,c+d+f-2,0,1).
\end{align}
By Lemmas~\ref{lem:2.2},~\ref{lem:2.3} and~\ref{lem:2.4}, we complete the proof.
\end{proof}

\begin{remark}
When $d=0$, the Witten zeta value
$\zeta_{\mathfrak{sl(4)}}(a,b,c,0,0,f)$ can also be viewed as a
Mordell-Tornheim sum with depth 3. The fact that every such sum can
be expressed as a rational linear combination of products of single
and double Euler sums when the weight $a+b+c+f$ is even has been
shown in~\cite{Tsu} and~\cite{ZBC}.
\end{remark}

\noindent\textbf{Acknowledgment.} The authors are grateful to the referee for carefully reading the manuscript and providing several constructive suggestions.

\end{document}